\documentclass[a4paper,reqno,12pt]{amsart}
\usepackage{amsmath,amscd,amsfonts,amssymb}
\usepackage{mathrsfs,dsfont}

\numberwithin{equation}{section}

\theoremstyle{definition}

\newtheorem*{remark*} {Remark}
\newtheorem*{definition*}{Definition}

\theoremstyle{plain}
\newtheorem*{theorem*}{Theorem}
\newtheorem{theorem}{Theorem}[section]
\newtheorem{lemma}[theorem]{Lemma}

\newtheorem{proposition}[theorem]{Proposition}
\newtheorem{example}[theorem]{Example}
\newtheorem{corollary}[theorem]{Corollary}
\newtheorem{definition}[theorem]{Definition}

\newtheorem*{theorem-m}{Theorem M}

\newtheorem{definition-main}{Definition}
\newtheorem{theorem-main}{Theorem}
\newtheorem{corollary-main}[theorem-main]{Corollary}
\newtheorem{proposition-main}[theorem-main]{Proposition}
\newtheorem{example-main}{Example}

\newcommand{\Z}{\mathbb{Z}}
\newcommand{\T}{\mathbb{T}}
\newcommand{\R}{\mathbb{R}}
\newcommand{\bz}{\mathbb{Z}}
\newcommand{\bt}{\mathbb{T}}
\newcommand{\br}{\mathbb{R}}

\newcommand{\norm}[1]{\left\lVert#1\right\rVert}
\newcommand{\1}{\mathds{1}}

\newcommand{\D}{\mathscr{D}}

\renewcommand{\geq}{\geqslant}
\renewcommand{\leq}{\leqslant}

\newcommand{\ft}[1]{\widehat #1}
\newcommand\dotprod[2]{\langle #1 , #2 \rangle}
\newcommand\mes{\operatorname{mes}}

\newcommand\inter{\operatorname{int}}

\newcommand{\sect}{Section~} 
\newcommand\eqdef{\overset{\text{def}}{=}}
\newcommand{\piped}{parallelepiped}
\newcommand\bmo{\operatorname{BMO}}

\def\Gam{\Gamma}
\def\lam{\lambda}
\def\Lam{\Lambda}

\newcommand{\va}{\ensuremath{\alpha}}

\newcommand{\ip}[2]{\langle #1, #2 \rangle}

\newenvironment{enumerate-math}
{\begin{enumerate}
\addtolength{\itemsep}{5pt}
}
{\end{enumerate}}

\newenvironment{enumerate-math-abc}
{\begin{enumerate}
\addtolength{\itemsep}{5pt}
}
{\end{enumerate}}

\begin{document}

\title{Riesz bases, Meyer's quasicrystals, and bounded remainder sets}
\author{Sigrid Grepstad}
\address{Institute of Financial Mathematics and Applied Number Theory, Johannes Kepler University Linz, Austria.}
\email{\texttt{sigrid.grepstad@jku.at}}
\author{Nir Lev}
\address{Department of Mathematics, Bar-Ilan University, Ramat-Gan 52900, Israel.}
\email{\texttt{levnir@math.biu.ac.il}}

\date{November 21, 2016}
\thanks{S.G.\ is supported by the Austrian Science Fund (FWF), Project F5505-N26, which is a part of the Special Research Program ``Quasi-Monte Carlo Methods: Theory and Applications''}
\thanks{N.L.\ is partially supported by the Israel Science Foundation grant No. 225/13}

\begin{abstract}
We consider systems of exponentials with frequencies belonging to simple quasi\-crystals in $\br^d$. 
We ask if there exist domains $S$ in $\br^d$ which admit such a system as a Riesz basis for the space $L^2(S)$.
We prove that the answer depends on an arithmetical condition on the quasicrystal. The proof is based on the connection of the problem to the discrepancy of multi-dimensional irrational rotations, and specifically, to the theory of bounded remainder sets. In particular it is shown that any bounded remainder set admits a Riesz basis of exponentials. This extends to several dimensions (and to the non-periodic setting) the results obtained earlier in dimension one.
\end{abstract}

\maketitle


\section{Introduction}

\subsection{Riesz bases}

Let $S$ be a bounded, measurable set in $\br^d$, and $\Lambda$ be a discrete set in $\br^d$.
In this paper we are interested in the Riesz basis property of the system of exponential functions
\begin{equation*}
E(\Lambda) = \left\{ e^{2\pi i \ip{\lambda}{x}} \right\}_{\lambda \in \Lambda}
\end{equation*}
in the space $L^2(S)$.

Recall that a system of vectors $\{f_n\}$ in a Hilbert space $H$ is a \emph{Riesz basis} if every $f \in H$ admits a unique
expansion $f = \sum c_n f_n$, with the coefficients $\{c_n\}$ satisfying
\begin{equation*}
A\|f\|^2 \leq \sum |c_n|^2 \leq B\|f\|^2
\end{equation*}
for some positive  constants $A$ and $B$ which do not depend on $f$.
It is well-known that this is equivalent to the system $\{f_n\}$ being simultaneously a \emph{frame} and a \emph{Riesz sequence} in the Hilbert space $H$
(see e.g.\ \cite{young}). 

The Riesz basis property of the exponential system $E(\Lambda)$ in the space $L^2(S)$ can be reformulated in terms of the Paley-Wiener space $PW_S$, consisting of all functions $f \in L^2(\br^d)$ whose Fourier transform 
\begin{equation*}
\ft{f}(t) = \int f(x) \, e^{-2\pi i \ip{t}{x}} dx
\end{equation*}
is supported by $S$. Namely, $E(\Lambda)$ is a Riesz basis in $L^2(S)$ if and only if $\Lambda$ is a \emph{complete interpolation set} for $PW_S$. The latter means that the interpolation problem
$f(\lambda) = c_{\lambda}$ $(\lambda \in \Lambda)$
admits a unique solution $f \in PW_S$ for every sequence $\{ c_{\lambda} \} \in \ell^2(\Lambda)$.

The construction of a Riesz basis of exponentials on a given set $S$ is generally a difficult
problem, and so far was achieved only in relatively few examples (see \cite{gre-lev-tiling, Kol15, NitKoz1, NitKoz2, LyuRash00, LyuSeip97, Mar06}).
 In particular it is not known whether the ball in dimensions two and higher
admits such a basis. On the other hand, no example is known of a set $S$ which does not have a Riesz basis of exponentials.

\subsection{Density}
A set $\Lambda \subset \br^d$ is called \emph{uniformly discrete} (or separated) if there is $\delta(\Lam) > 0$ such that 
$|\lam-\lam'| \geq \delta(\Lam)$ for any two distinct points $\lam, \lam' \in \Lam$. This condition is necessary for the system
$E(\Lambda)$ to be a Riesz basis in $L^2(S)$, and so will always be assumed below.  

An important role in the subject is played by the Beurling lower and upper uniform densities of a uniformly discrete set $\Lambda$, defined  respectively by
\begin{align*}
\D^-(\Lambda) & = \liminf_{R\to\infty} \,\inf_{x\in\R^d} \,\frac{\#(\Lambda\cap (x+B_R))}{|B_R|},\\[8pt]
\D^+(\Lambda)& =\limsup_{R\to\infty} \,\sup_{x\in\R^d} \,\frac{\#(\Lambda\cap (x+B_R))}{|B_R|},
\end{align*}
where $B_R$ denotes the ball of radius $R$ centered at the origin.
Landau \cite{lan} (see also \cite{nit-olev}) obtained necessary conditions for the system $E(\Lambda)$ to be a frame, or a Riesz sequence, in $L^2(S)$ in terms of these densities:

\medskip
\emph{If $E(\Lam)$ is a frame in $L^2(S)$, then $\D^-(\Lam) \geq \mes S$;}

\smallskip

 \emph{If $E(\Lam)$ is a Riesz sequence in $L^2(S)$, then $\D^+(\Lam) \leq \mes S$.}
\medskip

In the case when $S$ is a single interval $I \subset \R$, this result is due to Beurling \cite{beu} and Kahane \cite{kah}, who also proved that the condition
$\D^-(\Lam) > |I|$ is sufficient for $E(\Lambda)$ to be a frame in $L^2(I)$, while the condition $\D^+(\Lam) < |I|$ is sufficient for it  to be a Riesz sequence.
However, for disconnected sets $S$ and in the multi-dimensional case, sufficient conditions in terms of densities alone can not be given.

If the two densities $\D^-(\Lam)$ and $\D^+(\Lam)$ coincide, then their common value is called the \emph{uniform density} of the set $\Lam$, and will be denoted by $\D(\Lam)$. It follows from Landau's results above that:

\medskip
\emph{If $E(\Lambda)$ is a Riesz basis in $L^2(S)$, then $\Lambda$ has a uniform density $\D(\Lam) = \mes S$.}

\subsection{Universality}
It was discovered by Olevskii and Ulanovskii \cite{olevskii-ulanovskii:universal-1, olevskii-ulanovskii:universal-2, ou3}
 that there exist ``universal'' sets $\Lambda$, such that the system $E(\Lambda)$ is a frame (respectively a Riesz sequence) on any set $S$ of sufficiently small (respectively  large) measure:

\medskip
\emph{Given $a >0$ there is a uniformly discrete set $\Lam  \subset\br^d$, $\D(\Lam)=a$, such that:
\begin{enumerate-math}
\item \label{it:unsam}
$E(\Lambda)$ is a frame in $L^2(S)$ for any compact set $S$ with
$\mes S < \D(\Lam)$;
\item \label{it:unint}
$E(\Lambda)$ is a Riesz sequence in $L^2(S)$ for any open set $S$ with
$\mes S > \D(\Lam)$.
\end{enumerate-math}}
\medskip

In \cite{olevskii-ulanovskii:universal-1, olevskii-ulanovskii:universal-2, ou3} a  set $\Lam$ with the property \ref{it:unsam} was named
a ``universal sampling set'', while a set satisfying  \ref{it:unint} was called a ``universal interpolation set''
(the names are due to the role which such sets play in the theory of sampling and interpolation in Paley-Wiener spaces).
It was shown that such a set $\Lambda$ may be constructed by an arbitrarily small perturbation of a lattice in $\br^d$. It was also proved that the topological restrictions given on the set $S$ are indeed necessary -- if $S$ is allowed to be an arbitrary bounded measurable set, then no universal sampling or interpolation sets exist.

\subsection{Quasicrystals}
A different construction of universal sampling and interpolation sets, based on Meyer's ``cut-and-project'' method \cite{Mey70, Mey72}, was proposed by Matei and Meyer in \cite{meyer1, meyer2, meyer3}.
Let $\Gamma$ be a lattice in $\br^{d+1} = \br^d \times \br$, and let $p_1$ and $p_2$ denote the projections onto $\br^d$ and $\br$, respectively. Assume that the restrictions of $p_1$ and $p_2$ to $\Gamma$ are injective, and that their images are dense. Let $I$ be a semi-closed interval on $\R$,  $I = [a,b)$ or $I=(a,b]$, and consider the cut-and-project set $\Lam$ in $\br^d$ defined by
\begin{equation}
\label{eq:sqc}
\Lambda = \Lambda(\Gamma, I) = \{ p_1(\gamma) : \gamma \in \Gamma, \; p_2(\gamma) \in I\}.
\end{equation}
In \cite{meyer1, meyer3} such a set was named a ``simple quasicrystal''. It is well-known that $\Lambda$ is a uniformly discrete set, with uniform density
\[
\D(\Lambda) = \frac{|I|}{\det \Gam}.
\]

\begin{theorem-m}[Matei and Meyer]
If $\Lam$ is a simple quasicrystal defined by \eqref{eq:sqc} then it is a universal sampling and interpolation set, that is, it satisfies both \ref{it:unsam} and \ref{it:unint} above. 
\end{theorem-m}

In \cite{meyer2} the question was raised of what can be said in the ``critical case'' when the measure of the set $S$ is equal to the density of $\Lam$. In the one-dimensional periodic setting, this was analyzed in \cite{kozmalev}. The goal of the present paper is to extend the results obtained in \cite{kozmalev} to several dimensions and to the non-periodic setting.

\subsection{Results}
A bounded set $S \subset \br^d$ is called \emph{Riemann measurable} if its boundary has measure zero. Our first main result shows that ``most'' quasicrystals $\Lam$ do not provide a Riesz basis of exponentials $E(\Lambda)$ for any Riemann measurable set $S$. 
\begin{theorem}\label{thm2}
Let $\Lam$ be a simple quasicrystal defined by \eqref{eq:sqc} and such that
\begin{equation}
\label{eq:icn}
|I| \notin p_2(\Gamma).
\end{equation}
Then there exists no Riemann measurable set $S$ such that $E(\Lambda)$ is a Riesz basis in $L^2(S)$.
\end{theorem}

Hence there are only countably many possible values of the window length $|I|$ for which $E(\Lam)$ may serve as a Riesz basis.
 Our second main result shows that in the special case when $|I| \in p_2(\Gamma)$, the exponential system $E(\Lambda)$ indeed serves
as a Riesz basis for a family of sets $S$. To formulate the result we will need the following notion. 
\begin{definition*}
\label{def:equidecomp}
Two Riemann measurable sets $S$ and $S'$ in $\br^d$ are said to be \emph{equidecomposable} (or \emph{scissors congruent}) with respect to a group $G$ of rigid motions of $\br^d$, if the set $S$ can be partitioned into finitely many Riemann measurable subsets that can be reassembled using motions in $G$ to form, up to measure zero, a partition of $S'$. 
\end{definition*}
Equidecomposability is a classical notion dating back to Hilbert's third problem - the question of whether two polyhedra of equal volume are necessarily equidecomposable by rigid motions (see \cite{boltianski} for a detailed exposition of the subject).

\begin{theorem}\label{thm1}
Let $\Lam$ be a simple quasicrystal defined by \eqref{eq:sqc} and satisfying the condition
\begin{equation}
\label{eq:ic}
|I| \in p_2(\Gamma).
\end{equation}
Then $E(\Lam)$ is a Riesz basis in $L^2(S)$ for any Riemann measurable set $S$ such that
\begin{enumerate-math}
\item \label{it:dens}
$\mes S = \D(\Lam)$;
\item \label{it:equidecomp}
$S$ is equidecomposable to a {\piped} spanned by vectors in $p_1(\Gamma^*)$, using only translations by vectors in $p_1(\Gamma^*)$.
\end{enumerate-math}
\end{theorem}

Here we denote by $\Gamma^*$ the lattice dual to $\Gamma$ (see \sect \ref{sec:prelim}). 
Condition \eqref{eq:ic} ensures that the family of sets $S$ satisfying \ref{it:dens} and \ref{it:equidecomp} in Theorem \ref{thm1} is nonempty.
In fact, we will see that this family is in a sense ``dense'' among the sets of measure $\D(\Lambda)$:
\begin{proposition}
\label{prop:densefam}
Let $\Lam$ be a simple quasicrystal defined by \eqref{eq:sqc} and satisfying  \eqref{eq:ic}. 
Let $K$ be a compact set and $U$ be an open set
in $\br^d$, such that $K \subset U$ and $\mes K < \D(\Lambda) < \mes U$. Then one can find a Riemann measurable set $S$ such that $K \subset S \subset U$ and which satisfies  conditions \ref{it:dens} and \ref{it:equidecomp} in Theorem \ref{thm1}. 
\end{proposition}

The results above were outlined in \cite{gr-lev-cr}. 
 Special cases of Theorem \ref{thm1} were obtained in \cite{gre-lev-tiling, lev}. 
The present paper contains a detailed exposition and full proofs of the results.
In \sect \ref{sec:periodic} we also give analogous versions of the results in the periodic setting, where $S$ is a subset of the $d$-dimensional torus $\bt^d$, and the quasicrystal $\Lambda$ is a subset of $\bz^d$.

\subsection{Examples}
By particular choices of the lattice $\Gamma$ and the interval $I$ one can obtain more concrete versions of Theorem \ref{thm1}.

\begin{example}
\label{ex:exone}
Let $\alpha$ be an irrational number, and define
a sequence  $\Lam = \{\lam(n) \}$  by
\[
\lam(n) = n + \{ n \alpha\}, \quad n \in \Z
\]
(where $\{x\}$ denotes the fractional part of a real number $x$). Then the system $E(\Lam)$  is a Riesz basis in $L^2(S)$ for every set $S \subset \br$ which is the finite union of disjoint intervals with lengths in $\Z \alpha + \Z$ and of total length $1$.
\end{example}

\begin{example}
\label{ex:extwo}
The sequence $\Lam = \{\lam(n,m)\}$ defined by
\[
\lam(n,m) = (n,m) +  \{ n\sqrt 2 + m \sqrt 3 \} (\sqrt 2, \sqrt 3), \quad (n,m) \in \Z^2,
\]
provides a Riesz basis $E(\Lambda)$ in $L^2(S)$ for every set $S \subset \br^2$ which is equidecomposable to the unit square $Q = [0,1)^2$ using only translations by vectors in $\Z  (\sqrt 2, \sqrt 3) + \Z^2$.
\end{example}

These examples are special cases of Theorem \ref{thm:genex} below  (see \sect \ref{sec:genex}).

\subsection{Outline}
The proofs of the results above are based on three main ingredients. The first one is a key idea from \cite{meyer1, meyer2, meyer3} that we refer to as a ``duality'' principle, which  allows us to relate the Riesz basis property of $E(\Lambda)$ in $L^2(S)$ to the same property of another exponential system $E(\Lambda^*)$ in $L^2(I)$, where $I$ is the interval used in \eqref{eq:sqc} to define $\Lam$, and $\Lambda^*$ is a (non-simple) quasicrystal in $\R$ which is ``dual'' to $\Lambda$ (see \sect \ref{sect:dual}). 

This reduces the problem on exponential Riesz bases in $L^2(S)$ to a similar problem in $L^2(I)$, where $I$ is a single interval. The latter problem is much better understood due to availability of methods from the theory of entire functions, and we can use results of Avdonin \cite{avdonin} and Pavlov \cite{pavlov} that give conditions for $E(\Lambda^*)$ to be a Riesz basis on the interval $I$. This is the second main ingredient in our proofs.

To analyze the conditions from Avdonin and Pavlov's results we need our third main ingredient, which belongs to the theory of equidistribution and discrepancy for multi-dimensional irrational rotations. It is the theory of \emph{bounded remainder sets}, which in dimension one goes back to Hecke \cite{hecke}, Ostrowski \cite{ostrowski2, ostrowski} and Kesten \cite{kesten}. Using results from our recent paper \cite{gr-lev-brs} 
dealing with the multi-dimensional setting, we can prove that $E(\Lambda)$ is a Riesz basis on any bounded remainder set $S$ such that $\mes S = \D(\Lam)$.

The paper is organized as follows. \sect \ref{sec:prelim} contains some preliminary background.
In \sect \ref{sect:dual} the Matei-Meyer duality principle is explicitly formulated and proved. In \sect \ref{sec:special} we apply a linear change of variable to 
transform a general cut-and-project set to a canonical form which is more convenient to analyze.
In \sect \ref{sect:brs} we present relevant background on the concept of a bounded remainder set. The relation between this concept and one-dimensional cut-and-project sets is clarified in \sect \ref{sec:models-and-brs}. Finally in Sections \ref{sec:proof-thm-pos} and \ref{sec:proof-thm-neg} the main results are proved. The analogous results in the periodic setting are discussed in \sect \ref{sec:periodic}. In the last \sect \ref{sec:open} we mention some open problems.


\section{Preliminaries}
\label{sec:prelim}   

\subsection{Frames and Riesz sequences}
A system of vectors $\{f_n\}$ in a Hilbert space $H$ is called a \emph{frame} if there exist positive constants $A$ and $B$ such that the inequalities
\begin{equation}
\label{eq:framegen}
A \norm{f}^2 \leq \sum_n \left| \ip{f}{f_n} \right|^2 \leq B \norm{f}^2
\end{equation}
hold for all $f \in H$. The system $\{ f_n\}$ is called a \emph{Riesz sequence} if the inequalities
\begin{equation}
\label{eq:rseqgen}
A \sum_n \left|c_n \right|^2 \leq \Big\| \sum_n c_n f_n \Big\|^2 \leq B \sum_n \left| c_n \right|^2
\end{equation}
hold 
for every finite sequence of scalars $\{c_n\}$, for some positive constants $A$ and $B$ that do not depend on $\{c_n\}$.
The system $\{f_n\}$ is simultaneously a frame and a Riesz sequence if and only if it is a \emph{Riesz basis} in the Hilbert space $H$
(see \cite{young}).

If $S$ is a bounded, measurable set in $\br^d$, then 
the frame and Riesz sequence properties of the system of exponential functions $E(\Lambda)$ in the space $L^2(S)$ may be reformulated in terms of the sampling and interpolation properties of the set $\Lambda$ in the Paley-Wiener space $PW_S$. A discrete set $\Lambda \subset \br^d$ is called a \emph{set of sampling} for $PW_S$ if there are constants $A$ and $B$ such that 
\begin{equation*}
A \norm{f}_{L^2(\br^d)}^2 \leq \sum_{\lambda \in \Lambda} \left| f(\lambda) \right|^2 \leq B \norm{f}_{L^2(\br^d)}^2
\end{equation*} 
for all $f \in PW_S$. This means that a function $f \in PW_S$ can be reconstructed in a stable way from its samples $\{f(\lambda)\}$ on $\Lambda$. The set $\Lambda$ is called a \emph{set of interpolation} for $PW_S$ if the interpolation problem $f(\lam) = c_{\lam}$ has at least one solution $f \in PW_S$ for every sequence $\{c_{\lam}\} \in \ell^2(\Lam)$.
It is known (see \cite{young})  that $\Lambda$ is a set of  sampling for $PW_S$ if and only if the system $E(\Lambda)$ is a frame in the space $L^2(S)$, while the interpolation property of $\Lambda$ for the space $PW_S$ is equivalent to $E(\Lambda)$ being a Riesz sequence in $L^2(S)$.

The right hand side inequalities in \eqref{eq:framegen} and \eqref{eq:rseqgen} are automatically satisfied for the system $E(\Lambda)$ in $L^2(S)$
whenever $S$ is a bounded set and $\Lambda$ is a uniformly discrete set in $\br^d$  (see \cite{young}).
 Therefore, to show that $E(\Lambda)$ is a frame or Riesz sequence in $L^2(S)$, it is in this case enough to verify the left hand side inequalities in \eqref{eq:framegen} and \eqref{eq:rseqgen}. 

\subsection{Lattices}
By a (full-rank) \emph{lattice} $\Gamma \subset \br^k$ we mean the image of $\bz^k$ under an invertible $k \times k$ matrix $A$.
The determinant $\det (\Gamma)$ is equal to $| \det (A) |$. The \emph{dual lattice}  $\Gamma^*$  is the set of vectors $\gamma^* \in \br^k$ satisfying $\ip{\gamma}{\gamma^*} \in \bz$ for all $\gamma \in \Gamma$. Equivalently, $\Gamma^*$ is the image of $\bz^k$ under the matrix $A^{-\top}$, the inverse transpose of $A$.

\subsection{Model sets}
\label{sec:modelsets}
Let $\Gamma$ be a lattice in $\br^{n+m} = \br^n \times \br^m$, and let $p_1$ and $p_2$ denote the projections onto $\br^n$ and $\br^m$, respectively. Assume that the restrictions of $p_1$ and $p_2$ to $\Gamma$ are injective, and that their images are dense in $\br^n$ and $\br^m$, respectively. In this case we will say that $\Gamma$ is a lattice  in \emph{general position}.

Let $W$ be a bounded, Riemann measurable set in $\br^m$. Define a point set in $\br^n$ by
\begin{equation*}
\Lambda (\Gamma, W) := \left\{ p_1(\gamma) \, : \, \gamma \in \Gamma , \, p_2(\gamma) \in W \right\} .
\end{equation*}
Such a set is called a \emph{model set}, or a \emph{cut-and-project set}. These sets were introduced by Meyer in the beginning of 70's \cite{Mey70, Mey72}, and have been extensively studied as mathematical models for quasicrystals. 
The set $W$ is called the \emph{window} of the model set.

It is well-known that $\Lambda(\Gamma, W)$ is a uniformly discrete set, and has uniform density 
\begin{equation*}
\D(\Lambda(\Gamma, W)) =  \frac{\mes W}{\det \Gamma} 
\end{equation*}
(see for instance \cite[Proposition 5.1]{meyer3}).

\subsection{Simple quasicrystals}
The model set  $\Lambda(\Gamma, W)$ will be called a \emph{simple quasicrystal} if $m=1$ and if the window $W$ is a semi-closed interval $[a,b)$ or $(a,b]$.
This notion was introduced by Matei and Meyer in \cite{meyer1, meyer2, meyer3} where it was proved that a simple quasicrystal is a universal sampling and interpolation set. 

Remark that in these papers, the window $W$ was allowed to be also a closed interval $[a,b]$ or an open one $(a,b)$. Here, however, we define simple quasicrystals using only \emph{semi-closed} windows $W$, since otherwise this would affect the validity of Theorem \ref{thm1} in the case when both endpoints $a$, $b$ belong to $p_2(\Gamma)$.


\section{Duality}
\label{sect:dual}
\subsection{}
Let $\Gamma$ be a lattice in $\br^n \times \br^m$ in general position. Its dual lattice $\Gamma^*$ is then also in general position. Furthermore, let $U \subset \br^n$ and $V \subset \br^m$ be two bounded, Riemann measurable sets. There is a certain ``duality'' connecting the sampling and interpolation properties of the two model sets
\begin{align}
\label{eq:dualityeqs1} \Lambda(\Gamma, V) &= \{ p_1(\gamma) : \gamma \in \Gamma, \; p_2(\gamma) \in V\} \subset \br^n , \\
\label{eq:dualityeqs2} \Lambda^*(\Gamma, U) &= \{ p_2(\gamma^*) : \gamma^* \in \Gamma^*, \; p_1(\gamma^*) \in U\} \subset \br^m.
\end{align}
The following result was a key ingredient in Matei and Meyer's papers \cite{meyer1, meyer2, meyer3}, although it was not stated there explicitly in this form:
\begin{theorem}
\label{thm:duality} 
Suppose that the boundary of the set $V$ does not intersect $p_2(\Gamma)$. Then the following is true:
\begin{enumerate-math}
\item \label{item:frame} If $E(\Lambda^*(\Gamma, U))$ is a frame in $L^2(V)$, then $E(\Lambda(\Gamma, V))$ is a Riesz sequence in $L^2(U)$.
\item \label{item:riesz} If $E(\Lambda^*(\Gamma, U))$ is a Riesz sequence in $L^2(V)$, then $E(\Lambda(\Gamma, V))$ is a frame in $L^2(U)$.
\end{enumerate-math}
In the case when $\Lambda(\Gamma, V)$ is a simple quasicrystal (i.e.\ $m=1$ and $V$ is a semi-closed interval) the above is true regardless of whether or not the endpoints of $V$ lie in $p_2(\Gamma)$. 
\end{theorem}
This was used in the proof of Theorem~M to reduce the problem on exponential systems in $L^2(S)$ to a similar problem in $L^2(I)$, where $I$ is a single interval. Then the Beurling-Kahane results, which give sufficient conditions for the frame or Riesz sequence properties in terms of densities, allow to conclude the proof.

By combining \ref{item:frame} and \ref{item:riesz} of Theorem \ref{thm:duality}, we obtain the following:
\begin{corollary}
\label{cor:duality}
Under the same conditions as in Theorem \ref{thm:duality}, if the exponential system $E(\Lambda^*(\Gamma, U))$ is a Riesz basis in $L^2(V)$, then $E(\Lambda(\Gamma, V))$ is a Riesz basis in $L^2(U)$.
\end{corollary}
We will also use this duality to reduce the problem on exponential Riesz bases from $L^2(S)$ to $L^2(I)$. However, the latter problem can no longer be solved by density considerations, and it requires a more detailed analysis of the exponential system in question.

The remainder of this section is devoted to the proof of Theorem \ref{thm:duality}. Although this result is essentially contained in \cite{meyer1, meyer3}, we find it useful to include a detailed exposition of the proof for the specific formulation above.

\subsection{}
\label{sec:phidef}
We will need two auxiliary lemmas. Fix the two bounded, Riemann measurable sets $U \subset \br^n$ and $V \subset \br^m$, and choose an infinitely smooth, nonnegative function $\varphi$ on $\br^m$ with $\norm{\varphi}_{L^2} =1$ and support in the $m$-dimensional ball of radius $1$ around the origin. Moreover, in the special case when $m=1$ and $V$ is the semi-closed interval $[a,b)$ or $(a,b]$, we let $\varphi$ be supported in the interval $(0,1)$ or $(-1,0)$, respectively.

For $0<\varepsilon<1$, we define $\varphi_{\varepsilon}$ to be the function 
\begin{equation*}
\varphi_{\varepsilon}(x) = \frac{1}{\varepsilon^{m/2}} \, \varphi(x/\varepsilon).
\end{equation*}
We have $\norm{\varphi_{\varepsilon}}_{L^2}=1$, and the Fourier transform $\widehat{\varphi}_{\varepsilon}$ is given by
\begin{equation*}
\widehat{\varphi}_{\varepsilon} (t) = \varepsilon^{m/2} \, \widehat{\varphi} (\varepsilon t).
\end{equation*}

\begin{lemma}
\label{lemma:dualequidist}
Let $f$ be a Riemann integrable function on $U$. Then
\begin{equation*}
\lim_{\varepsilon \rightarrow 0} \sum_{\gamma^* \in \Gamma^*, \, p_1(\gamma^*) \in U} \left| f(p_1(\gamma^*) ) \widehat{\varphi}_\varepsilon (p_2(\gamma^*))\right|^2 = \det (\Gamma) \int_{U} \left| f(x) \right|^2 \, dx .
\end{equation*}
\end{lemma}
\begin{lemma}
\label{lemma:dualsupport}
Let $\{c(\gamma) \, : \, \gamma \in \Gamma \}$ be a sequence of complex numbers in $\ell^1(\Gamma)$. Then 
\begin{equation*}
\lim_{\varepsilon \rightarrow 0} \int_V \Big| \sum_{\gamma \in \Gamma} c(\gamma) \varphi_{\varepsilon}( t-p_2(\gamma)) \Big|^2 \, dt = \sum_{\gamma \in \Gamma , \, p_2(\gamma) \in V} |c(\gamma)|^2 , 
\end{equation*}
provided that $\partial V \cap p_2(\Gamma) = \emptyset$. 
If $m=1$ and $V$ is a semi-closed interval, the above is true regardless of whether or not the endpoints of $V$ lie in $p_2(\Gamma)$.
\end{lemma}
Proofs of these lemmas can basically be found in \cite{meyer3}.

\subsection{}
We can now give the proof of Theorem \ref{thm:duality}. 

\begin{proof}[Proof of part \ref{item:frame} of Theorem \ref{thm:duality}]
Suppose that $E(\Lambda^*(\Gamma, U))$ is a frame in $L^2(V)$. We will show that $E(\Lambda(\Gamma, V))$ is a Riesz sequence in $L^2(U)$. 
Let
\begin{equation*}
f(x) = \sum_{\gamma \in \Gamma} c(\gamma) \exp  2 \pi i \ip{p_1(\gamma)}{x} ,
\end{equation*}
where only finitely many coefficients $c(\gamma)$ are nonzero, and $c(\gamma) = 0$ whenever $p_2(\gamma) \notin V$. Since $\Lambda(\Gamma, V)$ is uniformly discrete, we must only show that
\begin{equation}
\label{eq:locdualfr}
\int_U |f(x)|^2 \, dx \geq C \sum_{\gamma \in \Gamma} |c(\gamma)|^2 
\end{equation}
for some constant $C$ (not depending on the sequence $\{ c(\gamma) \}$) . 

Consider the function
\begin{equation*}
G_{\varepsilon} (t) = \sum_{\gamma \in \Gamma} c(\gamma) \varphi_{\varepsilon} (t-p_2(\gamma)) . 
\end{equation*}
Notice that for sufficiently small $\varepsilon$, this function is supported on $V$. If $m=1$ and $V$ is a semi-closed interval, this is clear from the specific choice of support for $\varphi$. Otherwise, we assume that $\partial V \cap p_2(\Gamma) = \emptyset$. We then have $p_2(\gamma) \in \inter V$ whenever $c(\gamma) \neq 0$, and the same assertion follows.

Because $E(\Lambda^*(\Gamma, U))$ is a frame in $L^2(V)$, we have
\begin{equation}
\label{eq:locGeps}
C \int_V \left| G_{\varepsilon} (t) \right|^2 \, dt \leq \sum_{\gamma^* \in \Gamma^* , \, p_1(\gamma^*) \in U} \left| \widehat{G}_{\varepsilon} (p_2(\gamma^*))\right|^2 ,
\end{equation}
for some constant $C>0$. Now let $\varepsilon \rightarrow 0$. By Lemma \ref{lemma:dualsupport}, the left hand side of \eqref{eq:locGeps} tends to $C \sum_{\gamma \in \Gamma} |c(\gamma)|^2$. For the right hand side of \eqref{eq:locGeps}, we observe that 
\begin{equation*}
\widehat{G}_{\varepsilon} (p_2(\gamma^*)) = f(p_1(\gamma^*) ) \widehat{\varphi}_{\varepsilon} (p_2(\gamma^*)) .
\end{equation*}
Hence, Lemma \ref{lemma:dualequidist} applied to the function $f \cdot \1_U$ implies that the right hand side of \eqref{eq:locGeps} tends to $\det (\Gamma) \int_U |f(x)|^2 \, dx$.
This verifies \eqref{eq:locdualfr}, and concludes the proof of part \ref{item:frame} of Theorem \ref{thm:duality}.
\end{proof}

\begin{proof}[Proof of part \ref{item:riesz} of Theorem \ref{thm:duality}]
Now suppose that $E(\Lambda^*(\Gamma, U))$ is a Riesz sequence in $L^2(V)$. We will show that $E(\Lambda(\Gamma, V))$ is a frame in $L^2(U)$. Since $\Lambda(\Gamma, V)$ is uniformly discrete, it is sufficient to show that 
\begin{equation}
\label{eq:locframe}
\int_U |f(x)|^2 \, dx \leq C \sum_{\lambda \in \Lambda(\Gamma, V)} |\widehat{f}(\lambda)|^2
\end{equation}
for every $f \in L^2 (U)$ and some constant $C>0$ independent of $f$. Since $U$ is Riemann measurable, it is in fact sufficient to verify \eqref{eq:locframe} for any smooth $f$ supported on $U$. 

Given such $f$, define the function
\begin{equation*}
F_{\varepsilon}(t) = \sum_{\gamma^* \in \Gamma^*} f(p_1(\gamma^*)) \widehat{\varphi}_{\varepsilon}(p_2(\gamma^*)) \exp (2\pi i \ip{p_2(\gamma^*)}{t} ) .
\end{equation*}
This is an absolutely convergent trigonometric sum with nonzero coefficients only for frequencies in $\Lambda^*(\Gamma, U)$. 
Since $E(\Lambda^*(\Gamma, U))$ is a Riesz sequence in $L^2(V)$, we have
\begin{equation}
\label{eq:locFeps}
\sum_{\gamma^* \in \Gamma^*} \left| f(p_1(\gamma^*) ) \widehat{\varphi}_{\varepsilon} ( p_2(\gamma^*)) \right|^2 \leq C \int_V \left| F_{\varepsilon} (t) \right|^2 \, dt ,
\end{equation}
for some constant $C>0$. 

Now let $\varepsilon \rightarrow 0$. Since $f$ is supported by $U$, Lemma \ref{lemma:dualequidist} implies that the left hand side of \eqref{eq:locFeps} tends to $\det (\Gamma) \int_U |f(x)|^2 \, dx$. 
On the other hand, using Poisson's summation formula we can rewrite $F_{\varepsilon}$ as 
\begin{equation*}
F_{\varepsilon}(t) = \det (\Gamma) \sum_{\gamma \in \Gamma} \widehat{f} (p_1(\gamma)) \varphi_{\varepsilon} (t - p_2(\gamma)) .
\end{equation*}
When integrating $|F_{\varepsilon}(t)|^2$ over $V$ in \eqref{eq:locFeps}, we may restrict the summation to those terms for which $|p_2(\gamma)| < M$ for some sufficiently large $M>0$ (as the other terms are supported outside of $V$). Since $\widehat{f}$ is a Schwarz function, the coefficients $\{ \widehat{f} (p_1(\gamma)\}$ in this restricted sum belong to $\ell^1$. It thus follows from Lemma \ref{lemma:dualsupport} that the right hand side of \eqref{eq:locFeps} tends to the right hand side of \eqref{eq:locframe} as $\varepsilon \rightarrow 0$. This completes the proof of part \ref{item:riesz} of Theorem \ref{thm:duality}. 
\end{proof}


\section{Lattices in special form}
\label{sec:special}

\subsection{}
In this section we introduce a notion of lattices in \emph{special form} (see Definition \ref{def:speciallat}). 
Our motivation for introducing these lattices is that this allows us to simplify the discussion by considering bounded remainder sets only with respect to irrational rotations on $\bt^d = \br^d / \bz^d$, and avoid discussion of general $d$-dimensional torus groups.

We will show that any lattice in general position can be mapped onto a lattice of special form by a linear and invertible transformation. We can then restrict our attention to lattices of special form, and prove Theorems \ref{thm2} and \ref{thm1} for such lattices only. 
\begin{definition}
\label{def:speciallat}
Let $\Gamma$ be a lattice in $\br^d \times \br$. We say that $\Gamma$ (with dual $\Gamma^*$) is of \emph{special form} if 
\begin{align}
\label{eq:gam} \Gamma &= \{ ((\operatorname{Id}+\beta \alpha^{\top})m -\beta n, \, n-\alpha^{\top}m) : m \in \bz^d, n \in \bz \}, \\
\label{eq:gamdual} \Gamma^* &= \{ (m+\alpha n, \, (1+\beta^{\top}\alpha)n+\beta^{\top}m) : m \in \bz^d, n \in \bz \} ,
\end{align}
where $\operatorname{Id}$ denotes the $d \times d$ identity matrix, and $\alpha$, $\beta$ are column vectors in $\br^d$ satisfying the following conditions:
\begin{enumerate-math}
\item \label{it:alpha} The vector $\alpha = (\alpha_1, \alpha_2, \ldots , \alpha_d)^{\top}$ is such that the numbers $1 , \alpha_1, \alpha_2, \ldots , \alpha_d$ are linearly independent over the rationals.
\item \label{it:beta} The vector $\beta = (\beta_1 , \beta_2, \ldots , \beta_d)^{\top}$ is such that the numbers $\beta_1, \beta_2, \ldots , \beta_d, 1+ \beta^{\top} \alpha$ are linearly independent over the rationals.
\end{enumerate-math}
\end{definition}
Notice that the conditions imposed on the vectors $\alpha$ and $\beta$ are precisely those necessary and sufficient for the lattice $\Gamma$ and its dual $\Gamma^*$ to be in general position. This is most easily seen by considering the dual $\Gamma^*$. We have that 
\begin{equation*}
p_1(\Gamma^*) = \bz^d + \alpha \bz ,
\end{equation*}
and it is well-known that this set is dense in $\br^d$ if and only if the numbers $1, \alpha_1, \alpha_2, \ldots , \alpha_d$ are linearly independent over the rationals. This condition also guarantees that $p_1$ is injective when restricted to $\Gamma^*$. Similarly, we see that $p_2$ restricted to $\Gamma^*$ is injective if and only if the numbers $\beta_1, \beta_2, \ldots , \beta_d, 1+ \beta^{\top} \alpha$ are linearly independent over the rationals. The same condition implies that $p_2(\Gamma^*)$ is a dense set in $\br$. Thus, any lattice of special form is a lattice in general position. 

Notice that if $\Gamma$ is a lattice of special form, then $\det \Gamma = \det \Gamma^* = 1$.

We will now see that it is sufficient to prove Theorems \ref{thm2} and \ref{thm1} for lattices of the special form \eqref{eq:gam}, \eqref{eq:gamdual}. We begin by establishing some preliminary lemmas. 

\subsection{}
Let $A$ be an $n \times n$ invertible matrix, and $B$ be an $m \times m$ invertible matrix. These determine a linear and invertible transformation $T$ from $\br^n \times \br^m$ to itself given by
\begin{equation}
\label{eq:transt}
T : (x, y) \mapsto (Ax, By) , 
\end{equation}
where $x \in \br^n$, $y \in \br^m$. Let $\Gamma$ and $L$ be two lattices in general position in $\br^n \times \br^m$, and let $U \subset \br^n$ and $V \subset \br^m$ be two bounded, Riemann measurable sets. 
\begin{lemma}
\label{lem:lattices}
Assume that $T$ maps $L$ onto $\Gamma$. Then the following are equivalent:
\begin{enumerate-math}
\item $E(\Lambda (L, V))$ is a Riesz basis in $L^2(U)$.
\item $E(\Lambda (\Gamma, BV))$ is a Riesz basis in $L^2(A^{-\top} U)$.
\end{enumerate-math}
\end{lemma}
\begin{proof}
Since $\Gamma = T(L)$, we have that 
\begin{equation*}
\Lambda (\Gamma, BV) = \left\{ p_1(\gamma) \, : \, \gamma \in \Gamma , \, p_2(\gamma) \in BV \right\} = \left\{ A p_1(l) \, : \, l \in L , \, p_2(l) \in V  \right\} .
\end{equation*}
Hence, the set $\Lambda (\Gamma, BV)$ is the image of $\Lambda (L, V)$ under the linear and invertible transformation given by $A$. The result thus follows from the fact that for any point set $\Lambda \subset \br^n$, the exponential system $E(\Lambda)$ is a Riesz basis in $L^2(U)$ if and only if $E(A \, \Lambda)$ is a Riesz basis in $L^2(A^{-\top} U)$. 
\end{proof}
We remark that Lemma \ref{lem:lattices} remains true if the words Riesz basis are replaced by frame or Riesz sequence.

\subsection{}
We now restrict our attention to lattices in $\br^d \times \br$.
\begin{lemma}
\label{lem:equivclasses}
Let $L \subset \br^d \times \br$ be a lattice in general position. Then one can find a lattice $\Gamma$ of special form \eqref{eq:gam} and a linear and invertible transformation $T$ as in \eqref{eq:transt} such that $T(L)=\Gamma$.
\end{lemma}

\begin{proof}
Rather than showing that there exists a transformation of the form \eqref{eq:transt} mapping $L$ onto $\Gamma$, we will prove the equivalent claim that there exists a transformation of the form \eqref{eq:transt} mapping $L^*$ onto $\Gamma^*$. To see that these are indeed equivalent, observe that if the transformation $(x,y) \mapsto (Ax, By)$ maps $L$ onto $\Gamma$, then the dual lattice $\Gamma^*$ is the image of $L^*$ under the transformation $(x,y) \mapsto (A^{-\top}x, B^{-\top}y)$.

The lattice $L^*$ is the image of $\bz^{d+1}$ under a linear and invertible transformation. Let this transformation be represented by the matrix $M$, with 
\begin{equation*}
M (m,n) = (am + bn , c^{\top}m + en) , \quad m \in \bz^d , \, n \in \bz , 
\end{equation*}
where $a$ is a $d \times d$ matrix, $b$ and $c$ are $d \times 1$ vectors, and $e$ is a scalar. 

Let $T$ be the transformation in \eqref{eq:transt} with $A = a^{-1}$ and $B = 1/ (e - c^{\top} a^{-1} b)$. The fact that the set 
\begin{equation*}
p_1(L^*) = \left\{ am + bn \, : \, m \in \bz^d , \, n \in \bz \right\}
\end{equation*}
is dense in $\br^d$ guarantees that the matrix $a$ is invertible, so $A$ is well-defined. The scalar $B$ is also well-defined, since $e - c^{\top} a^{-1} b = \det M / \det a \neq 0$. One can check that for this choice of $T$ we have that $T(L^*) = \Gamma^*$, where $\Gamma^*$ is given in \eqref{eq:gamdual} with $\alpha := A b$ and $\beta := B c$. 
Finally, since $p_1(\Gamma^*) = Ap_1(L^*)$ and $p_2(\Gamma^*)=Bp_2(L^*)$, and $L^*$ is in general position, it follows that also $\Gamma^*$ must be in general position. This in turn implies that the vectors $\alpha$, $\beta$ must satisfy the conditions in Definition \ref{def:speciallat}.
\end{proof}

\subsection{}
With Lemmas \ref{lem:lattices} and \ref{lem:equivclasses} established, let us now use these to show that we can restrict our attention to lattices of special form when proving Theorems \ref{thm2} and \ref{thm1}. 

\begin{lemma}
\label{lem:thm2spec}
If Theorem \ref{thm2} is true for lattices $\Gamma$ of the special form \eqref{eq:gam}, then it is true for any lattice in general position.
\end{lemma}
\begin{proof}
Assume that Theorem \ref{thm2} holds for any lattice of the special form \eqref{eq:gam}. Let $L$ be a lattice in general position, and let $I$ be an interval. Suppose that the quasicrystal $\Lambda(L, I)$ provides a Riesz basis of exponentials in $L^2(U)$ for some Riemann measurable set $U$. We will show that this implies $|I| \in p_2(L)$. 

By Lemma \ref{lem:equivclasses}, there exists a linear and invertible transformation $T$ as in \eqref{eq:transt} mapping $L$ onto a lattice $\Gamma$ of special form \eqref{eq:gam}. By Lemma \ref{lem:lattices}, the set of exponentials $E(\Lambda (\Gamma, BI))$ is a Riesz basis in $L^2(A^{-\top} U)$, and the set $A^{-\top} U$ is Riemann measurable. Since Theorem \ref{thm2} holds for the lattice $\Gamma$, this implies that $|BI| \in p_2(\Gamma)$. Finally, observe that since $T(L) = \Gamma$, we have $p_2(\Gamma)=Bp_2(L)$, and thus $|I| \in p_2(L)$.
\end{proof} 

\begin{lemma}
\label{lem:thm1spec}
If Theorem \ref{thm1} is true for lattices $\Gamma$ of the special form \eqref{eq:gam}, then it is true for any lattice in general position.
\end{lemma}
\begin{proof}
Assume that Theorem \ref{thm1} holds for any lattice of the special form \eqref{eq:gam}. Let $L$ be a lattice in general position, and let $I$ be an interval satisfying the condition $|I| \in p_2(L)$. Denote by $S$ a Riemann measurable set with $\mes S = \D(\Lambda(L, I))$, which is equidecomposable to a {\piped} spanned by vectors in $p_1(L^*)$ using only translations by vectors in $p_1(L^*)$. We will show that $E(\Lambda (L, I))$ is a Riesz basis in $L^2(S)$.

By Lemma \ref{lem:equivclasses}, there exists a linear and invertible transformation $T$ as in \eqref{eq:transt} mapping $L$ onto a lattice $\Gamma$ of special form \eqref{eq:gam}. We have that $Bp_2(L) = p_2(\Gamma)$, and thus the condition $|I| \in p_2(L)$ implies that $|BI| \in p_2(\Gamma)$. Since Theorem \ref{thm1} holds for the lattice $\Gamma$, it follows that $E(\Lambda (\Gamma, BI))$ is a Riesz basis in $L^2(U)$ for any set $U$, with $\mes U = \D(\Lambda (\Gamma, BI))$, which is equidecomposable to a {\piped} spanned by vectors in $p_1(\Gamma^*)$ using only translations by vectors in $p_1(\Gamma^*)$. Hence, if we can show that the set $A^{-\top} S$ satisfies these two conditions, then the proof will be concluded by Lemma \ref{lem:lattices}.

Let us first verify that $\mes A^{-\top} S = \D(\Lambda (\Gamma, BI))$. To see this, observe that $T(L)=\Gamma$ implies $|B \det A| \det L = \det \Gamma$, and hence
\begin{equation*}
\D(\Lambda (\Gamma, BI)) = \frac{|BI|}{\det \Gamma} = \frac{|I|}{|\det A |\det L } = \frac{\D( \Lambda (L,I))}{|\det A|}.
\end{equation*}
Since $\mes S = \D( \Lambda (L, I))$, we get $\mes A^{-\top} S = \D( \Lambda (L,I))/ |\det A| = \D( \Lambda (\Gamma, BI))$.

Let us now see that $A^{-\top} S$ satisfies the appropriate equidecomposability condition.
Recall that if $T(L)=\Gamma$, then the dual lattice $\Gamma^*$ is the image of $L^*$ under the transformation $(x,y) \mapsto (A^{-\top}x, B^{-\top}y)$. In particular, the matrix $A^{-\top}$ sends any vector in $p_1(L^*)$ to a vector in $p_1(\Gamma^*)$. It follows that $A^{-\top}$ maps any {\piped} spanned by vectors in $p_1(L^*)$ (respectively, any set equidecomposable to such a {\piped} using translations by vectors in $p_1(L^*)$) to a {\piped} spanned by vectors in $p_1(\Gamma^*)$ (respectively, a set equidecomposable to such a {\piped} using translations by vectors in $p_1(\Gamma^*)$). Hence, the set $A^{-\top} S$ is equidecomposable to a {\piped} spanned by vectors in $p_1(\Gamma^*)$ using only translations by vectors in $p_1(\Gamma^*)$. This completes the proof of Lemma \ref{lem:thm1spec}.
\end{proof}

\subsection{}
In a similar way, we can show the same for Proposition \ref{prop:densefam}. 

\begin{lemma}
\label{lem:cor1spec}
If Proposition \ref{prop:densefam} is true for lattices $\Gamma$ of the special form \eqref{eq:gam}, then it is true for any lattice in general position.
\end{lemma}
\begin{proof}
Assume that Proposition \ref{prop:densefam} holds for any lattice of the special form \eqref{eq:gam}. Let $L$ be a lattice in general position, and let $I$ be an interval satisfying $|I| \in p_2(L)$. Given any open set $U \subset \br^d$ and compact set $K$, $K \subset U$, with $\mes K < \D(\Lambda(L,I)) < \mes U$, we want to find a set $S$, $K \subset S \subset U$, where $S$ satisfies the conditions in Theorem \ref{thm1}.

By Lemma \ref{lem:equivclasses}, there exists a linear and invertible transformation $T$ as in \eqref{eq:transt} mapping $L$ onto a lattice $\Gamma$ of special form \eqref{eq:gam}. We have seen in the proof of Lemma \ref{lem:thm1spec} that this implies $A^{-\top}p_1(L^*) = p_1(\Gamma^*)$. In light of this, consider the open set $A^{-\top}U$ and the compact set $A^{-\top} K$, $A^{-\top} K \subset A^{-\top} U$, satisfying $\mes A^{-\top} K < \D(\Lambda(L,I))/ |\det A| < \mes A^{-\top} U$. Since $\D(\Lambda (L,I))/|\det A| = \D(\Lambda (\Gamma, BI))$, and since Proposition \ref{prop:densefam} holds for the lattice $\Gamma$, we can find a set $V$, $A^{-\top} K \subset V \subset A^{-\top} U$, where $\mes V = \D(\Lambda (\Gamma, BI))$ and $V$ is equidecomposable to a {\piped} spanned by vectors in $p_1(\Gamma^*)$ using only translations by vectors in $p_1(\Gamma^*)$.

Now let $S = A^{\top} V$. Then $K \subset S \subset U$ and $\mes S = \D(\Lambda (L,I))$. Moreover, since $A^{\top} p_1(\Gamma^*) = p_1(L^*)$, the set $S$ is equidecomposable to a {\piped} spanned by vectors in $p_1(L^*)$ using only translations by vectors in $p_1(L^*)$. Thus, the set $S$ satisfies the conditions in Theorem \ref{thm1}.
\end{proof}


\section{Bounded remainder sets}
\label{sect:brs}

In this section we give a brief introduction to \emph{bounded remainder sets} in $\br^d$, and mention their role in our problem. In the next section, this will be used to analyze the distribution of points in one-dimensional model sets. 

\subsection{}
Let $\alpha \in \br^d$ be a vector such that the numbers $1, \alpha_1 , \alpha_2 , \ldots , \alpha_d$ are linearly independent over the rationals. It is well-known that under this condition, the sequence $\{ n\alpha \}$ is equidistributed on the $d$-dimensional torus $\bt^d = \br^d / \bz^d$, which means that 
\begin{equation}
\label{eq:equidist}
\frac{1}{n} \sum_{k=0}^{n-1} \chi_S (x+k\alpha) \rightarrow \mes S  \quad ( n \rightarrow \infty )
\end{equation}
for any $x \in \bt^d$ and every Riemann measurable set $S \subset \bt^d$. Here, $\chi_S$ denotes the indicator function for $S$. 
One can also consider $S$ as a set in $\br^d$, in which case $\chi_S$
should be understood as the multiplicity function for the projection of $S$ on $\bt^d$, that is
\begin{equation*}
\chi_S(x) = \sum_{k \in \bz^d} \1_S(x+k). 
\end{equation*}

A quantitative measure of the equidistribution of the sequence $\{ n \alpha \}$ is given by the \emph{discrepancy function}
\begin{equation}
\label{eq:discset}
D_n(S,x) = \sum_{k=0}^{n-1} \chi_S (x+k\alpha) - n \mes S. 
\end{equation}
By \eqref{eq:equidist}, we have $D_n(S,x) = o(n)$, $n \rightarrow \infty$, for any Riemann measurable set $S \subset \br^d$. 

However, the discrepancy obeys an even stricter bound for certain special sets $S$. We say that $S \subset \br^d$ is a \emph{bounded remainder set} (BRS) if there exists a constant $C = C(S, \alpha)$ such that $|D_n(S, x)| \leq C$ for every $n$ and almost every $x$. 
The classical example is when $S$ is a single interval in dimension one. In this case, it was shown by Hecke \cite{hecke} and Ostrowski \cite{ostrowski2, ostrowski} that if the length of the interval belongs to $\bz \va + \bz$, then it is a BRS. Kesten \cite{kesten} proved that this condition is also necessary for an interval to be a BRS. 

The relevance of bounded remainder sets to the subject of this paper is clarified by the following:
\begin{theorem}
\label{thm3}
Let $\Lambda = \Lambda (\Gamma, I)$ be the simple quasicrystal defined in \eqref{eq:sqc}, where $\Gamma$ is a lattice of special form \eqref{eq:gam}. Then $E(\Lambda)$ is a Riesz basis in $L^2(S)$ for every Riemann measurable bounded remainder set $S$ with $\mes S = |I|$.
\end{theorem} 
When we say that $S$ is a bounded remainder set, we mean with respect to the vector $\alpha$ in the definition of the special lattice $\Gamma$.

\subsection{}
With exception of the one-dimensional case, the problem of explicitly describing bounded remainder sets has until recently remained quite open. Sz\"{u}sz gave the first non-trivial examples of bounded remainder sets in two dimensions in 1954 by 
constructing a family of parallelograms of bounded remainder \cite{szusz}.
Liardet later generalized Sz\"{u}sz' construction to all dimensions \cite{liardet}. 

In our recent paper \cite{gr-lev-brs}, a comprehensive study of multi-dimensional bounded remainder sets was done. First we extended to higher dimensions the Hecke-Ostrowski result on intervals.
\begin{theorem}
\label{thm:piped}
Let $P$ be a {\piped} in $\br^d$, spanned by vectors $v_1, \ldots , v_d$ belonging to $\bz \va + \bz^d$. Then $P$ is a bounded remainder set.
\end{theorem}
This result guarantees the existence of a large collection of bounded remainder sets, which in particular
encompasses the examples previously given by Sz\"{u}sz and Liardet.

On the other hand, we also proved that the Riemann measurable bounded remainder sets can be characterized by equidecomposability to a {\piped} of the above form.
\begin{theorem}
\label{thm:decomp}
Let $S \subset \br^d$ be a Riemann measurable set. Then $S$ is a bounded remainder set if and only if there is a {\piped} $P$ spanned by vectors belonging to $\bz \va + \bz^d$, such that $S$ and $P$ are equidecomposable (by Riemann measurable pieces) using only translations by vectors in $\bz \va + \bz^d$. 
\end{theorem}

It is not difficult to show that if two sets $S$ and $S'$ are equidecomposable using
translations by vectors in $\bz \va + \bz^d$, and if one of them is a BRS, then so is
the other (see \cite[Proposition 4.1]{gr-lev-brs}). We  proved in \cite{gr-lev-brs} that  also the converse is true:

\begin{theorem}
\label{thm:anydecomp}
Let $S$ and $S'$ be two Riemann measurable bounded remainder sets of the same measure. Then $S$ and $S'$ are equidecomposable using translations by vectors in $\bz \va + \bz^d$ only.
\end{theorem}
Hence, if $S$ is a Riemann measurable BRS, then $S$ is equidecomposable to \emph{any} {\piped} $P$ spanned by vectors belonging to $\bz \va + \bz^d$,
such that $\mes P =\mes S$.

It is known (see \cite[Proposition 2.4]{gr-lev-brs}) that the measure of any bounded remainder set must be of the form 
\begin{equation}
\label{eq:allowedmes}
n_0 + n_1 \va_1 + \cdots + n_d \va_d ,
\end{equation}
where $n_0 , \ldots , n_d$ are integers. Conversely, we have the following result.
\begin{theorem}
\label{thm:realized}
Any positive number $\gamma$ of the form \eqref{eq:allowedmes} can be realized as the measure of some bounded remainder {\piped} spanned by vectors belonging to $\bz \va + \bz^d$.
\end{theorem}
This follows from Theorem \ref{thm:piped} and \cite[Proposition 3.7]{gr-lev-brs}.

\subsection{}
\label{sec:bdsmes}

We complete this section by showing that the bounded remainder sets are, in a certain sense, dense among the sets of a given measure in $\br^d$. The following theorem is essentially Proposition \ref{prop:densefam} for lattices of special form.
\begin{theorem}
\label{thm:densmes} 
Let $\gamma$ be a positive number of the form \eqref{eq:allowedmes}. Suppose that $U \subset \br^d$ is an open set, $K$ is compact, $K \subset U$, and $\mes K < \gamma < \mes U$. Then there exists a Riemann measurable bounded remainder set $S$, $K \subset S \subset U$, such that $\mes S = \gamma$.
\end{theorem}
\begin{proof}
We can assume that the set $U$ is bounded. If not, let $U_R$ be the intersection of $U$ with the ball of radius $R$ centered at the origin. For a sufficiently large $R$, we have $K \subset U_R$ and $\mes U_R > \gamma$, and we may thus continue with $U_R$ in place of $U$. 

We first construct two bounded remainder sets $A$ and $B$ satisfying $K \subset A \subset B \subset U$ and $\mes A < \gamma < \mes B$. The set $\bz \va + \bz^d$ is dense in $\br^d$, so for any $\varepsilon >0$ we can find by Theorem \ref{thm:piped} a bounded remainder {\piped} $P_{\varepsilon}$ of diameter smaller than $\varepsilon$ spanned by vectors in $\bz \alpha + \bz^d$. Consider a tiling of $\br^d$ by translated copies of $P_{\varepsilon}$. Let $A$ be the union of all {\piped}s intersecting $K$, and $B$ be the union of those contained in $U$. Then $A$ and $B$ are bounded remainder sets.
Choosing $\varepsilon$ sufficiently small, we can guarantee that $K \subset A \subset B \subset U$ and $\mes A < \gamma < \mes B$.

We complete the proof by showing that there exists a bounded remainder set $S$ satisfying $A \subset S \subset B$ and $\mes S = \gamma$. Since $A$ and $B$ are both bounded remainder sets and $A \subset B$, their difference $B \setminus A$ is also a BRS. Theorem \ref{thm:realized} ensures that we can construct two disjoint {\piped}s, $P$ and $Q$, spanned by vectors in $\bz \va + \bz^d$, where $\mes P = \gamma - \mes A$ and $\mes Q = \mes B - \gamma$. Their union $P \cup Q$ is a BRS of measure equal to that of $B \setminus A$, and by Theorem \ref{thm:anydecomp} the sets $P \cup Q$ and $B \setminus A$ are equidecomposable using translations by vectors in $\bz \va + \bz^d$. It follows that $P$ is equidecomposable to some subset $R \subset B \setminus A$, and by Theorem \ref{thm:decomp} the set $R$ is a BRS. Now let $S = A \cup R$. The set $S$ is a BRS satisfying $A \subset S \subset B$, and $\mes S = \mes A + \mes P = \gamma$.
\end{proof}


\section{Model sets and bounded remainder sets}
\label{sec:models-and-brs}

In this section we study the distribution of points in a one-dimensional model set. We assume that the window of the model set is a Riemann measurable bounded remainder set (with respect to the projected lattice). We show that in this case, the model set can be obtained by a bounded perturbation of an arithmetic progression, and moreover the perturbations are of the same size on the average. These results will allow us later on to apply the theorem of Avdonin in the proof of Theorem \ref{thm1}.

\subsection{}
Let $\Gamma$ be a lattice in $\br^d \times \br$, and let $S$ be a bounded, Riemann measurable set in $\br^d$. In this section we study the distribution of points in the one-dimensional model set $\Lambda := \Lambda^*(\Gamma, S)$ defined by \eqref{eq:dualityeqs2}. We may restrict ourselves to lattices $\Gamma$ of the special form \eqref{eq:gam} 
(in view of the results in \sect \ref{sec:special}, the general case can be reduced to this one by applying a linear transformation). Then using the expression \eqref{eq:gamdual} for the dual lattice $\Gamma^*$, one can check that in this case the model set is given by
\begin{equation}
\label{eq:lambdastar}
\Lambda = \left\{ n+ \dotprod{n \alpha + m}{\beta} \, : \, n \in \bz, \, m \in \bz^d, \, n\alpha +m \in S \right\},
\end{equation}
where $\alpha$ and $\beta$ are the vectors used to define $\Gamma$. Notice that  $\Lambda$ has uniform density
\[
\D(\Lambda) = \mes S.
\]

\subsection{}
\label{sec:bddev}
Recall that $S$ is a \emph{bounded remainder set} (BRS) with respect to $\alpha$ if there is a constant $C = C(S,\alpha)$, such that  the discrepancy $D_n(S,x)$ defined by \eqref{eq:discset} satisfies the condition $|D_n(S, x)| \leq C$ for every $n$ and almost every $x$. In this case, one may arrange this
condition to hold for all $x$ in a given countable set, by replacing $S$ with an appropriate translation $S+t$ (for this matter almost every $t$ will do).
Let us assume that the discrepancy is bounded for all the points of the form $\{j\va\}$, which amounts to the condition
\begin{equation}
\label{eq:posnegn}
\sup_{n>0} \, \sup_{j\in\bz} \, \Big| \sum_{k=j+1}^{j+n} \chi_S(k\alpha) - n \mes S \Big| < \infty.
\end{equation}
Remark that in the converse direction, a Riemann measurable set $S$ which satisfies the condition \eqref{eq:posnegn} must be a bounded remainder set, see \cite[Proposition 2.2]{gr-lev-brs}.

\begin{lemma}
\label{lem:bddist}
Assume that condition \eqref{eq:posnegn} is satisfied. Then the model set \eqref{eq:lambdastar} can be enumerated as a sequence $\{\lambda_j\}$, $j \in \bz$, in such a way that
\begin{equation}
\label{eq:delta_j_bdd}
\sup_{j \in \Z} \Big| \lambda_j - \frac{j}{\mes S} \Big| < \infty.
\end{equation}
\end{lemma}
In other words, the model set $\Lambda$ can be obtained by a bounded perturbation of the points in the arithmetic progression $(1/\mes S)\bz$. The property of a cut-and-project set being at bounded distance from a lattice has been considered by some authors, see e.g.\ \cite{oguey, haynes, senechal} and the references therein.

\begin{proof}[Proof of Lemma \ref{lem:bddist}]
Define
\begin{equation}
\label{eq:enum}
S_n := S \cap (n\alpha + \bz^d) , \quad \Lambda_n := \{ n + \dotprod{x}{\beta} \, : \, x \in S_n \} , \quad n \in \bz.  
\end{equation}
One can see from \eqref{eq:lambdastar}  that the sets $\{ \Lambda_n \}$ form a partition of $\Lambda$ (it is not excluded that some of the $\Lambda_n$ are empty). 
Let $\{ s_n \}$, $n \in \bz$, be a sequence of integers such that 
\begin{equation}
\label{eq:sn-def}
s_{n+1} - s_n = \# \Lambda_n,
\end{equation}
and choose an enumeration $\{ \lambda_j \, : \, j \in \bz \}$ of the set $\Lambda$ in such a way that
\begin{equation}
\label{eq:enum2}
\Lambda_n = \{ \lambda_j \, : \, s_n \leq j < s_{n-1}\}.
\end{equation}
We will show that condition \eqref{eq:delta_j_bdd} is satisfied for this enumeration.

By \eqref{eq:enum}, \eqref{eq:sn-def} we have $s_{n+1} - s_n = \chi_S(n\alpha)$. Hence \eqref{eq:posnegn} implies that for $n >0$,
\begin{equation}
\label{eq:snbound1}
s_n = s_0 + \sum_{k=0}^{n-1} \chi_S(k\alpha)  = n \mes S + O(1),
\end{equation}
and in a similar way one can see that the same is true also for $n\leq0$.

Now given $j$, there is $n =n(j)$ such that $s_n \leq j < s_{n+1}$, and so $\lambda_j \in \Lambda_n$. We have
\begin{equation}
\label{eq:displace}
\lambda_j - \frac{j}{\mes S} = (\lambda_j - n) + \left( n - \frac{s_n}{\mes S} \right) + \left( \frac{s_n-j}{\mes S} \right).
\end{equation}
Since $S$ is a bounded set, there is a constant $R$ such that $\Lambda_n \subset [n-R, n+R]$ for every $n \in \bz$. Hence the first term on the right hand side of \eqref{eq:displace} is bounded. The second term is also bounded, due to \eqref{eq:snbound1}. Finally, the third term is bounded as well, since the number of elements in each set $\Lambda_n$ is bounded. Thus, we obtain \eqref{eq:delta_j_bdd}.
\end{proof}

\subsection{}
\label{sec:avgdev}
It is well-known that a set $S$ is a bounded remainder set if and only if there exists a bounded, measurable function $g$ on the $d$-dimensional torus $\bt^d = \br^d / \bz^d$ such that
\begin{equation}
\label{eq:cohom}
\chi_S (x) - \mes S = g(x+\va)-g(x)   \quad \text{a.e.}
\end{equation}
A simple proof of this fact can be found in \cite[Proposition 2.3]{gr-lev-brs}. The equation \eqref{eq:cohom} is known as the \emph{cohomological equation} for the function $\chi_S$. The function $g$ is unique a.e.\ up to an additive constant, and is called the \emph{transfer function} for $S$. 

We proved in \cite{gr-lev-brs} that if the bounded remainder set $S$ is Riemann measurable, then the transfer function $g$ may be chosen to be a Riemann integrable function:
\begin{theorem}[see {\cite[Theorem 6]{gr-lev-brs}}]
\label{thm:cohom}
Let $S$ be a Riemann measurable bounded remainder set. Then there is a bounded, Riemann integrable function $g: \bt^d \rightarrow \br$ satisfying \eqref{eq:cohom}.
\end{theorem}

The proof of this result is based on the characterization of the Riemann measurable bounded remainder sets given in Theorems \ref{thm:piped} and \ref{thm:decomp} above.

By applying an appropriate translation to the set $S$, we may arrange the equality \eqref{eq:cohom} to hold for all the points $x$ of the form $\{n\va\}$, that is,
\begin{equation}
\label{eq:cohom-zero}
\chi_S (n \va) - \mes S = g((n+1)\va)-g(n \va), \quad n \in \Z.
\end{equation}
In fact, since \eqref{eq:cohom} holds for almost every $x$, almost every translation of $S$ will satisfy the above.
Notice that condition \eqref{eq:cohom-zero} and the boundedness of $g$ imply \eqref{eq:posnegn}. 

\subsection{}
Let $\Lambda = \{\lambda_j\}$ be the enumeration given by Lemma \ref{lem:bddist}. Then by \eqref{eq:delta_j_bdd} we have
\begin{equation}
\label{eq:delta_j}
\sup_{j\in\bz} |\delta_j|<\infty, \quad \text{where} \quad \delta_j := \lambda_j - \frac{j}{\mes S} .
\end{equation}
We will now see that the perturbations $\delta_j$ are in fact of the same size on the average:

\begin{lemma}
\label{lem:avgdev}
Assume that there is a Riemann integrable function $g$ satisfying \eqref{eq:cohom-zero}. Then there exists a constant $c$ such that
\begin{equation}
\label{eq:limitval}
 \sup_{k \in \bz} \Big|\frac{1}{N} \sum_{j=k+1}^{k+N} \delta_j - c \, \Big| \to 0 \quad (N \to \infty).
\end{equation}
\end{lemma}

\begin{proof}
We continue to use the same notations introduced in the proof of Lemma \ref{lem:bddist}. First we are going to derive a simple expression for the sum $\sum \delta_j$ with $j$ going through the ``block'' $s_n \leq j < s_{n+1}$. Indeed, we have
\begin{equation*}
\sum_{s_n \leq j < s_{n+1}} \delta_j = \sum_{s_n \leq j < s_{n+1}} \left( \lambda_j - n \right) - \sum_{s_n \leq j < s_{n+1}} \left( \frac{j}{\mes S} - n \right) \eqdef S_1(n) - S_2(n).
\end{equation*}
We evaluate each one of the sums $S_1(n)$ and $S_2(n)$ separately. Consider the function
\begin{equation*}
\phi(x) := \sum_{m \in \bz^d} \ip{x+m}{\beta} \, \1_S(x + m).
\end{equation*}
This function is $1$-periodic, hence it may  be viewed as a function on $\bt^d= \br^d/\bz^d$. By \eqref{eq:enum}, \eqref{eq:enum2}  we  have $S_1(n) = \phi(n\alpha)$. The second sum $S_2(n)$ can be calculated explicitly, 
\begin{equation}
\label{eq:s2sum}
S_2(n) = (s_{n+1}-s_n) \left( \frac{s_{n+1}+s_n-1}{2\mes S} - n \right).
\end{equation}
Using condition \eqref{eq:cohom-zero} we get that for $n>0$,
\begin{equation}
\label{eq:sng}
s_n = s_0 + \sum_{k=0}^{n-1} \chi_S(k\alpha)  = n \mes S + g(n \alpha) + c_1,
\end{equation}
and similarly the same is true also for $n \leq 0$ (with the same constant $c_1$). 
Substituting this expression for $s_n$ in \eqref{eq:s2sum} yields that $S_2(n) = \psi(n\va)$, where
\[
\psi(x) :=  \frac{ \chi_S(x) (g(x) + g(x+\va) + c_2)}{2\mes S}.
\]
We conclude that 
\begin{equation*}
\sum_{s_n \leq j < s_{n+1}} \delta_j = h(n \va), \quad n \in \bz,
\end{equation*}
where $h : \bt^d \rightarrow \br$ is the Riemann integrable function given by $h(x) := \phi(x) - \psi(x)$.

Now to prove \eqref{eq:limitval} it will be enough to consider the case where $k = s_n - 1$ and $k+N = s_{n+r} - 1$, that is, where the sum in \eqref{eq:limitval} goes though $r$ consecutive ``blocks''. This is due to the fact that the $\delta_j$ are known to be bounded by \eqref{eq:delta_j}, and the number of elements in each ``block'' is bounded as well. So, consider the sum
\begin{equation*}
\sum_{s_n \leq j < s_{n+r}} \delta_j = \sum_{m=n}^{n+r-1} \sum_{s_m \leq j < s_{m+1}} \delta_j = \sum_{m=n}^{n+r-1} h(m\va).
\end{equation*}
The points $\{ m\alpha \}$ are well-equidistributed on $\bt^d$, and $h$ is Riemann integrable, hence
\begin{equation*}
\sup_{n \in \bz} \Big| \sum_{m=n}^{n+r-1} h(m\va) - r \int_{\bt^d} h(x) \, dx \Big| = o(r) , \quad r \rightarrow \infty
\end{equation*}
(see \cite[pp.\ 46, 52]{kuipers}). From \eqref{eq:sng} it follows that $s_{n+r} - s_n =r \mes S + O(1)$, and hence
\begin{equation*}
\frac{1}{s_{n+r} - s_n} \sum_{s_n \leq j < s_{n+r}} \delta_j = \frac{1}{\mes S} \int_{\bt^d} h(x) \, dx + o(1) , \quad r \rightarrow \infty ,
\end{equation*}
uniformly with respect to $n$. It follows that \eqref{eq:limitval} holds with
\begin{equation*}
c = \frac{1}{\mes S} \int_{\T^d} h(x) \, dx,
\end{equation*}
and this concludes the proof of Lemma \ref{lem:avgdev}. 
\end{proof}


\section{Model sets that give Riesz bases}
\label{sec:proof-thm-pos}
We are now equipped to present the proof of Theorem \ref{thm1}. The result states that a simple quasicrystal $\Lambda(\Gamma, I)$ for which $|I| \in p_2 (\Gamma)$ provides a Riesz basis of exponentials in $L^2(S)$ for any Riemann measurable set $S$ satisfying the two conditions
\begin{enumerate-math}
\item 
$\mes S = \D(\Lam)$;
\item 
$S$ is equidecomposable to a {\piped} spanned by vectors in $p_1(\Gamma^*)$, using only translations by vectors in $p_1(\Gamma^*)$.
\end{enumerate-math}
By Lemma \ref{lem:thm1spec} it will be enough to consider the case when $\Gamma$ and $\Gamma^*$ are of the special form \eqref{eq:gam}, \eqref{eq:gamdual}. Then the quasicrystal $\Lambda (\Gamma, I)$ has density $\D(\Lambda) = |I|$, and
\begin{equation*}
p_1(\Gamma^*) = \bz \va + \bz^d .
\end{equation*}
Hence, Theorem \ref{thm1} follows immediately from Theorems \ref{thm3} and \ref{thm:decomp}. It remains to prove Theorem \ref{thm3}.

\subsection{Avdonin's theorem}
We will prove Theorem \ref{thm3} by invoking the duality in Corollary \ref{cor:duality}. Namely, in order to prove that $E(\Lambda(\Gamma, I))$ is a Riesz basis in $L^2(S)$, it is sufficient to show that $E(\Lambda^*(\Gamma,S))$ is a Riesz basis in $L^2(I)$. For the latter we will use the following result due to Avdonin, which gives a sufficient condition on a system of exponential functions to be a Riesz basis in $L^2(I)$.
\begin{theorem}[Avdonin \cite{avdonin}]
\label{thm:avdonin}
Let $I \subset \br$ be an interval, and let $\{ \lambda_j , \, j \in \bz \}$ be a sequence in $\br$ satisfying the following three conditions:
\begin{enumerate-math-abc}
\item
\label{item:separation}
$\{\lambda_j\}$ is a separated sequence, that is, $\inf_{j \neq k} |\lambda_j - \lambda_k| > 0$;
\item
\label{item:boundedness}
$\sup_{j} |\delta_j| < \infty$, where $\delta_j := \lambda_j - j/|I|$;
\item
\label{item:cancellation}
There is a constant $c$ and a positive integer $N$ such that
\begin{equation}
\label{eq:avdonin-condition}
\sup_{k \in \Z} \; \Big| \frac1{N} \sum_{j=k+1}^{k+N} \delta_{j} \; - \; c \, \Big| < \frac1{4 |I|} \, .
\end{equation}
\end{enumerate-math-abc}
Then the system $\{ e^{2\pi i \lambda_j} \}$ is a Riesz basis in $L^2(I)$.
\end{theorem}
This is a generalization of Kadec's $1/4$ theorem, which corresponds to the case $N=1$. In fact, the theorem above is a special case of the result given in \cite{avdonin}.

\subsection{Proof of Theorem \ref{thm3}}
Let $\Lambda(\Gamma, I)$ be the simple quasicrystal \eqref{eq:sqc} with $\Gamma$ of special form \eqref{eq:gam}, and let $S$ be a Riemann measurable bounded remainder set with $\mes S = |I|$. We want to show that $E(\Lambda(\Gamma, I))$ is a Riesz basis in $L^2(S)$. By Theorem \ref{thm:cohom} there exists a Riemann integrable function $g$ satisfying the cohomological equation \eqref{eq:cohom}. Moreover, by translation of $S$ we may assume that conditions \eqref{eq:posnegn} and \eqref{eq:cohom-zero} are satisfied. Such a translation will not affect the Riesz basis property for $S$. 

By Corollary \ref{cor:duality} with $U=S$ and $V=I$, it will be sufficient to show that the dual system $E(\Lambda^*(\Gamma, S))$ is a Riesz basis in $L^2(I)$. 
As condition \eqref{eq:posnegn} is satisfied, we can invoke Lemma \ref{lem:bddist} to obtain an enumeration of $\Lambda^*(\Gamma, S)$ for which condition \ref{item:boundedness} in Theorem \ref{thm:avdonin} holds. Moreover, since \eqref{eq:cohom-zero} is satisfied, Lemma \ref{lem:avgdev} guarantees that also condition \ref{item:cancellation} holds for this enumeration. Finally, condition \ref{item:separation} is satisfied as well since $\Lambda^*(\Gamma, S)$ is a uniformly discrete set. Thus, $E(\Lambda^*(\Gamma, S))$ is a Riesz basis in $L^2(I)$ by Theorem \ref{thm:avdonin}, and this completes the proof of Theorem \ref{thm3}.

\subsection{Proof of Proposition \ref{prop:densefam}}
We now show that if $|I| \in p_2(\Gamma)$, then the simple quasicrystal $\Lambda (\Gamma, I)$ provides a Riesz basis $E(\Lambda(\Gamma, I))$ in $L^2(S)$ for a large collection of sets $S$ in the following sense: Given any compact set $K$ and open set $U$, where $K \subset U$ and $\mes K < \D(\Lambda) < \mes U$, one can find a Riemann measurable set $S$ satisfying the conditions in Theorem \ref{thm1} such that $K \subset S \subset U$. 

Indeed we have already seen that Proposition \ref{prop:densefam} is true when $\Gamma$, $\Gamma^*$ are lattices of the special form \eqref{eq:gam}, \eqref{eq:gamdual}; this follows from Theorems \ref{thm:decomp} and \ref{thm:densmes}. 
By Lemma \ref{lem:cor1spec}, the proposition is true also in the general case.

\subsection{Examples}
\label{sec:genex}
Finally, let us see how Examples \ref{ex:exone} and \ref{ex:extwo} can be deduced from Theorem \ref{thm3}. These examples are special cases of the following more general result. 
\begin{theorem}
\label{thm:genex}
Let $\alpha$ and $\beta$ be column vectors in $\br^d$ satisfying conditions \ref{it:alpha} and \ref{it:beta} in Definition \ref{def:speciallat}, and define a sequence $\Lambda = \{ \lambda (m) \}$ by
\begin{equation}
\label{eq:specqc}
\lambda (m) = m + \{ \alpha^{\top} m \} \beta , \quad m \in \bz^d 
\end{equation}
(where $\{ x \}$ denotes the fractional part of a real number $x$). Then the  system $E(\Lambda)$ is a Riesz basis in $L^2(S)$ for every Riemann measurable set $S$ which is equidecomposable to the unit cube $Q = [0,1)^d$ using translations by vectors in  $\bz \va + \bz^d$.
\end{theorem}
\begin{proof}
Let $I$ be the interval $(-1,0]$, and let $\Gamma$ and $\Gamma^*$ be lattices of the special form \eqref{eq:gam}, \eqref{eq:gamdual}. Then the simple quasicrystal $\Lambda = \Lambda (\Gamma, I)$ is the set \eqref{eq:specqc}. By Theorem \ref{thm3}, $E(\Lambda)$ is a Riesz basis in $L^2(S)$ for every Riemann measurable bounded remainder set $S$ with $\mes S = 1$. Hence by Theorem \ref{thm:decomp}, $E(\Lambda)$ is a Riesz basis in $L^2(S)$ for every Riemann measurable set $S$ which is equidecomposable to the unit cube using translations by vectors in $\bz \va + \bz^d$.
\end{proof}

Examples \ref{ex:exone} and \ref{ex:extwo} follow directly from Theorem \ref{thm:genex}. In the one-dimensional case, if we take $\alpha \in \br$ irrational and $\beta=1$, then we obtain Example \ref{ex:exone}. In two dimensions, Example \ref{ex:extwo} follows by choosing $\alpha = \beta = (\sqrt{2}, \sqrt{3})$.


\section{Model sets which do not give Riesz bases}
\label{sec:proof-thm-neg}
In this section we prove Theorem \ref{thm2}. That is, we show that if the simple quasicrystal $\Lambda (\Gamma, I)$ does not satisfy the arithmetical condition $|I| \in p_2(\Gamma)$, then there is no Riemann measurable set $S \subset \br^d$ such that $E(\Lambda(\Gamma, I))$ is a Riesz basis in $L^2(S)$.

\subsection{Bounded mean oscillation}
Let us recall the definition of functions and sequences with bounded mean oscillation. Let $f(x)$ be a locally integrable function on $\br^d$, and denote by $f_J$ the average of $f$ over a bounded interval $J \subset \br$, that is $f_J = |J|^{-1} \int_J f(x) \, dx$. The \emph{mean oscillation} of $f$ over $J$ is defined as
\begin{equation*}
\frac{1}{|J|} \int_J \left| f(x) - f_J \right| \, dx.
\end{equation*}
If the mean oscillation of $f$ is bounded uniformly over all intervals $J$, then we say that $f$ has \emph{bounded mean oscillation}, and we write $f \in \bmo (\br)$. Clearly any bounded function belongs to $\bmo (\br)$, but it is well-known that $\bmo (\br)$ contains also unbounded functions, such as the function $f(x) = \log |x|$.

Similarly, one can define the space $\bmo(\bz)$ of sequences with bounded mean oscillation. We say that a sequence of complex numbers $\{c_n\}_{n \in \bz}$ has bounded mean oscillation, and we write $\{c_n\} \in \bmo(\bz)$, if 
\begin{equation*}
\sup_{n<m} \left( \frac{1}{m-n} \sum_{k=n+1}^{m} \left| c_k - \frac{c_{n+1} + \cdots + c_{m}}{m-n} \right| \right)  
\end{equation*} 
is finite. 

\subsection{Discrepancy function}
For a discrete set $\Lambda \subset \br$, denote by $n_{\Lambda}(x)$ the \emph{counting function} for $\Lambda$ satisfying 
\begin{equation*}
n_{\Lambda}(y) - n_{\Lambda}(x) = \# \left( \Lambda \cap [x,y) \right), \quad x < y .
\end{equation*}
This condition defines $n_{\Lambda}(x)$ uniquely up to an additive constant. 

If the set $\Lambda$ has uniform density $\D (\Lambda)$, then we define the \emph{discrepancy function} of $\Lambda$ to be the difference
\begin{equation}
\label{eq:onedimdisc}
d(\Lambda, x) = n_{\Lambda}(x) - \D(\Lambda) x , \quad x \in \br .
\end{equation}
This piecewise linear function with slope $-\D(\Lambda)$ and positive unit jumps at every $x \in \Lambda$ gives a quantitative measure of the uniform distribution of $\Lambda$. From the definition of $\D(\Lambda)$ it is clear that $d(\Lambda, x) = o(x)$ as $x \rightarrow \pm \infty$.

Now let $\Lambda = \Lambda^*(\Gamma, S)$ as given in \eqref{eq:lambdastar}, and consider the associated discrepancy function $d(\Lambda, x)$. There is a close connection between $d(\Lambda, x)$ and the sequence $\{D_n(S)\}_{n \in \bz}$ defined as 
\begin{equation}
\label{eq:d_ns}
D_n(S) = 
\begin{cases}
\sum_{k=0}^{n-1} \chi_S(k\alpha) - n \mes S , \quad &n >0 \\
0 , \quad &n=0\\
-\sum_{k=n}^{-1} \chi_S(k\alpha) - n \mes S , \quad &n<0 \\
\end{cases}.
\end{equation}
In Lemma \ref{lem:bddist} we looked at the case when this sequence is bounded, and showed that $\Lambda$ is then at bounded distance from an arithmetical progression. One can check that in this case, the corresponding discrepancy $d(\Lambda, x)$ is a bounded function.

We now consider the case when $d(\Lambda,x)$ belongs to $\bmo (\br)$.
\begin{lemma}
\label{lem:bmorel}
Let $\Lambda = \Lambda^*(\Gamma, S)$ be given in \eqref{eq:lambdastar}, and suppose that $d(\Lambda,x) \in \bmo (\br)$. Then $\{ D_n(S) \}$ in \eqref{eq:d_ns} belongs to $\bmo(\bz)$.
\end{lemma}
\begin{proof}
We introduce a new function $\tilde{n}_{\Lambda}(x)$, defined by $\tilde{n}_{\Lambda} (0) = 0$ and the condition 
\begin{equation}
\label{eq:tilden}
\tilde{n}_{\Lambda}(y)-\tilde{n}_{\Lambda}(x) = \sum_{k \in [x,y) \cap \bz} \# \Lambda_k, \quad x<y,
\end{equation}
where $\Lambda_k$ is a block in the partition of $\Lambda$ given in \eqref{eq:enum}. 
We may think of $\tilde{n}_{\Lambda}(x)$ as the counting function for a multiset with multiplicity $\# \Lambda_k$ at the point $x=k$. 
Recall that $\Lambda_k \subset [k-R, k+R]$ for some $R=R(S, \Gamma)$ and that the block sizes $\# \Lambda_k$ are uniformly bounded. It follows that $n_{\Lambda}(x)-\tilde{n}_{\Lambda}(x)$ is a bounded function. Thus, if we define $f(x)$ as
\begin{equation*}
f(x) = \tilde{n}_{\Lambda}(x) - x \mes S,
\end{equation*}
then from $\D (\Lambda) = \mes S$ it follows that the difference $d(\Lambda, x) - f(x)$ is also a bounded function. Since $d(\Lambda, x)$ belongs to $\bmo (\br)$, so does $f(x)$.

The function $f(x)$ is piecewise linear, with slope $-\mes S$ and bounded integer jumps at integer values of $x$. We can therefore write $f$ as the sum of two functions, $f=g+h$, where $g(x)$ is piecewise constant and equal to $f(n)$ on $(n-1,n]$, and $h(x)$ is $1$-periodic and linear with slope $-\mes S$ on each such interval. The mean oscillation of the function $g$ over the interval $[n,m)$ is given by
\begin{equation}
\label{eq:bmog}
\frac{1}{m-n} \sum_{k=n+1}^m \left| f(k) - \frac{f(n+1) + \cdots + f(m)}{m-n} \right| . 
\end{equation}
Since $g$ differs from $f$ by a bounded function, we have $g \in \bmo(\br)$, so \eqref{eq:bmog} is bounded uniformly with respect to $n$ and $m$. In other words, the sequence $\{f(n) \}_{n \in \bz}$ belongs to $\bmo(\bz)$. 
Finally we have that $f(n) = D_n(S)$, and hence $\{ D_n(S)\} \in \bmo(\bz)$. 
\end{proof}

\subsection{Pavlov's theorem}
To prove Theorem \ref{thm2}, we use the duality in Theorem \ref{thm:duality} to transfer our problem from $L^2(S)$ to $L^2(I)$. We will show that for $E(\Lambda^*(\Gamma, S))$ to be a Riesz basis in $L^2(I)$, it is necessary that $|I| \in p_2(\Gamma)$, and by Corollary \ref{thm:duality} this will imply Theorem \ref{thm2}. 
As what we need is a necessary, and not a sufficient, condition for $E(\Lambda^*(\Gamma, S))$ to be a Riesz basis in $L^2(I)$, we cannot use Avdonin's theorem. Instead we will use the following consequence of Pavlov's complete characterization of the exponential Riesz bases in $L^2(I)$ \cite{pavlov}.
\begin{theorem}[See {\cite[Theorem 8, p.\ 240]{hruscev-nikolskii-pavlov}}]
\label{thm:pavlov}
Let $\Lambda \subset \br$ be a discrete set. Then for the exponential system $E(\Lambda)$ to be a Riesz basis in $L^2(0,a)$, $a>0$, it is necessary that the function $f(x) = n_{\Lambda}(x) -ax$ belongs to $\bmo(\br)$.
\end{theorem}

\subsection{Proof of Theorem \ref{thm2}}
By Lemma \ref{lem:thm2spec}, it will be enough to consider the case when $\Gamma$, $\Gamma^*$ are lattices of the special form \eqref{eq:gam}, \eqref{eq:gamdual}. Since the boundary of the set $S$ has measure zero, there exists a translate of $S$ whose boundary does not intersect the countable set $p_1(\Gamma^*)$. Translating $S$ does not affect the Riesz basis property for the set, so we may assume below that $\partial S \cap p_1(\Gamma^*) = \emptyset$.

Suppose that $E(\Lambda(\Gamma, I))$ is a Riesz basis in $L^2(S)$. By applying Corollary \ref{cor:duality} with the lattice
\begin{equation*}
\Gamma' := \left\{ (p_2(\gamma^*), p_1(\gamma^*)) \, : \, \gamma^* \in \Gamma^* \right\} \subset \br \times \br^d
\end{equation*}
and with $U=I$ and $V=S$, it follows that $E(\Lambda^*(\Gamma, S))$ is a Riesz basis in $L^2(I)$, with $\Lambda^*(\Gamma, S)$ given by \eqref{eq:lambdastar}.

Denote by $d(\Lambda^*, x)$ the discrepancy function for $\Lambda^*=\Lambda^*(\Gamma, S)$. Since $E(\Lambda^*)$ is a Riesz basis in $L^2(I)$, it follows from Landau's necessary density conditions that $\D(\Lambda^*) = \mes S = |I|$. Thus, by Theorem \ref{thm:pavlov} we have $d(\Lambda^*,x) \in \bmo(\br)$, and from Lemma \ref{lem:bmorel} it follows that the sequence $\{D_n(S)\}$ in \eqref{eq:d_ns} belongs to $\bmo(\bz)$. We now need the following result to complete the proof.
\begin{theorem}[\cite{kozmalev}]
\label{thm:bmomes}
Let $\alpha \in \br^d$ be an irrational vector, and $S \subset \br^d$ be a Riemann measurable set. 
If the sequence $\{ D_n(S)\}_{n \in \bz}$ in \eqref{eq:d_ns} belongs to $\bmo(\bz)$, then the measure of $S$ is of the form 
\begin{equation}
\label{eq:bdsmes}
n_0 + n_1 \alpha_1 + \cdots + n_d \alpha_d , 
\end{equation}
where $n_0, n_1, \ldots , n_d$ are integers. 
\end{theorem}
This result was proved in \cite[Section 4]{kozmalev} in the one-dimensional case. 
The proof in higher dimensions is along the same line. Indeed, consider the function
\begin{equation*}
 f(x) = \chi_S(x)-\mes S,
\end{equation*}
which is a Riemann integrable function on $\bt^d$. By assumption, the ergodic sums
\begin{equation*}
 S_n(x) := \sum_{k=0}^{n-1} f(x+k\alpha)
\end{equation*}
satisfy the condition $\{S_n(0)\} \in \bmo$. Hence, as in the proof of \cite[Theorem 4.3]{kozmalev},
 it follows that there is a real-valued function $g \in L^2(\bt^d)$ such that 
$f(x) = g(x+\alpha)-g(x)$ almost everywhere. In turn, the
proof of \cite[Proposition 2.4]{gr-lev-brs} implies that $\mes S$ is of the form \eqref{eq:bdsmes}
(notice that in the latter proof the function $g$ was bounded, but this fact was not used
in the proof -- only the measurability of $g$ is important).

Finally, we observe that when $\Gamma$ is given by \eqref{eq:gam}, then $p_2(\Gamma)$ is precisely the collection of real numbers of the form \eqref{eq:bdsmes}. As $|I|=\mes S$, we thus get $|I| \in p_2(\Gamma)$, and this completes the proof of Theorem \ref{thm2}.


\section{The periodic setting}
\label{sec:periodic}

\subsection{}
There is also a version of the problem in the periodic setting, where $S$ is a Riemann measurable subset of $\bt^d$, and the simple quasicrystal is a subset of $\bz^d$.

Let $\alpha \in \br^d$ be a vector such that the numbers $1, \alpha_1, \alpha_2, \ldots , \alpha_d$ are linearly independent over the rationals, and let $I$ be a semi-closed interval on the circle $\bt=\br/\bz$. Then the set  
\begin{equation}
\label{eq:pqc}
\Lam (\alpha, I) := \{ n \in \bz^d \, : \, \ip{n}{\alpha} \in I \} 
\end{equation}
is called a \emph{simple quasicrystal} in $\bz^d$. 

One can check that $\Lambda (\alpha, I)$ has uniform density $\D(\Lam (\alpha, I)) = |I|$. 

The result analogous to Theorem~M  in this setting is the following \cite{meyer2}:

\smallskip

\emph{
Let $\Lambda$ be a simple quasicrystal defined by \eqref{eq:pqc}. Then:
\begin{enumerate-math}
\item
$E(\Lambda)$ is a frame in $L^2(S)$ for any compact set $S \subset \bt^d$ with
$\mes S < |I|$;
\item
$E(\Lambda)$ is a Riesz sequence in $L^2(S)$ for any open set $S \subset \bt^d$ with
$\mes S > |I|$.
\end{enumerate-math}}

\subsection{}
For Riesz bases we have the following versions of Theorems \ref{thm2} and \ref{thm1}.
\begin{theorem}
\label{pthm2}
Let $\Lambda$ be a simple quasicrystal defined by \eqref{eq:pqc}, and suppose that $|I|$ is not of the form 
\begin{equation*}
n_0 + n_1 \alpha_1 + \cdots + n_d \alpha_d ,
\end{equation*}
where $n_0, \ldots , n_d$ are integers. Then there is no Riemann measurable set $S \subset \bt^d$ such that $E(\Lambda)$ is a Riesz basis in $L^2(S)$.
\end{theorem}
\begin{theorem}
\label{pthm1}
Let $\Lambda$ be a simple quasicrystal defined by \eqref{eq:pqc}, and suppose that 
\begin{equation*}
|I| = n_0 + n_1 \alpha_1 + \cdots + n_d \alpha_d 
\end{equation*}
for certain integers $n_0, \ldots , n_d$. Then $E(\Lambda)$ is a Riesz basis in $L^2(S)$ for every Riemann measurable bounded remainder set $S \subset \bt^d$ with $\mes S = |I|$.
\end{theorem}
As before, when we say that $S$ is a bounded remainder set, we mean with respect to the vector $\alpha$.

\subsection{}
As in the non-periodic case, there is a duality connecting the frame and Riesz sequence properties of $E(\Lambda(\alpha, I))$ to those of the ``dual'' quasicrystal in $\bz$ defined by
\begin{equation}
\label{eq:pdualqc}
\Lam^*(\alpha, S) := \{ m \in \bz \, : \, -m \alpha \in S \}.
\end{equation} 
This duality can be stated in a form similar to Theorem \ref{thm:duality}. By combining its two parts we obtain the following analog of Corollary \ref{cor:duality}.
\begin{samepage}
\begin{lemma}
\label{lem:pduality} 
\quad 
\begin{enumerate-math}
\item If the exponential system $E(\Lam^*(\alpha, S))$ is a Riesz basis in $L^2(I)$, then $E(\Lam (\alpha, I))$ is a Riesz basis in $L^2(S)$.
\item Suppose that the boundary of $S$ does not intersect the set $\bz \alpha$. If $E(\Lam(\alpha, I))$ is a Riesz basis in $L^2(S)$, then $E(\Lam^*(\alpha, S))$ is a Riesz basis in $L^2(I)$.
\end{enumerate-math}
\end{lemma}
\end{samepage}
This result allows us to reduce the problem on exponential Riesz bases in $L^2(S)$ to a similar problem in $L^2(I)$ for the interval $I \subset \bt$, and again we can apply the results of Avdonin \cite{avdonin} and Pavlov \cite{pavlov} to verify Theorems \ref{pthm2} and \ref{pthm1}. 
We will not present this in detail. A full proof in the one-dimensional periodic case is given in \cite{kozmalev}.


\section{Remarks. Open problems}
\label{sec:open}
Finally we mention some problems which are left open.

\subsection{}
Suppose that the simple quasicrystal $\Lambda(\Gamma, I)$ satisfies the arithmetical condition $|I| \in p_2(\Gamma)$. Which sets $S$ will then admit $E(\Lambda(\Gamma,I))$ as a Riesz basis? It is enough to restrict our attention to lattices $\Gamma$ of special form \eqref{eq:gam}. We have then seen in Theorem \ref{thm3} that a sufficient condition is that $S$ is a Riemann measurable bounded remainder set with $\mes S = |I|$. Is this condition also necessary?

This question is related to a problem in discrepancy theory. In the proof of Theorem \ref{thm2} we saw that a necessary condition for $E(\Lambda (\Gamma, I))$ to be a Riesz basis in $L^2(S)$ is that the sequence of discrepancies
\begin{equation}
\label{eq:bmoaex}
\left\{ \sum_{k=0}^{n-1} \chi_S (k\alpha) - n \mes S \, : \, n =1,2, 3 \ldots \right\}
\end{equation}
is in $\bmo$. It is an open question whether there exists a set $S$ for which the sequence \eqref{eq:bmoaex} is unbounded, but is in $\bmo$. In the simplest case when $S$ is a single interval in dimension one, the answer to this question is negative \cite{kozmalev}. If the answer is negative also in the general case, then the bounded remainder property not only suffices, but in fact characterizes the Riemann measurable sets $S$ for which $E(\Lambda (\Gamma, I))$ is a Riesz basis in $L^2(S)$.

\subsection{}
In this paper we have studied the Riesz basis property for $E(\Lambda)$ when $\Lambda = \Lambda (\Gamma, I)$ is a simple quasicrystal. The duality in Theorem \ref{thm:duality} allows us to reduce the problem to that of determining when the quasicrystal $\Lambda^* = \Lambda^*(\Gamma, S)$ provides a Riesz basis of exponentials in $L^2(I)$, where $I$ is an interval. This is a problem which is far better understood, and where powerful tools from the theory of entire functions apply.

Recall that the duality in Theorem \ref{thm:duality} is in fact twofold; it says that $E(\Lambda)$ is a frame in $L^2(S)$ if $E(\Lambda^*)$ is a Riesz sequence in $L^2(I)$, and that $E(\Lambda)$ is a Riesz sequence in $L^2(S)$ if $E(\Lambda^*)$ is a frame in $L^2(I)$. Rather than using both statements simultaneously to determine when $E(\Lambda)$ is a Riesz basis in $L^2(S)$ (as we do in Theorems \ref{thm2} and \ref{thm1}), one may apply parts \ref{item:frame} and \ref{item:riesz} of Theorem \ref{thm:duality} separately. 
Seip and Ortega-Cerd\`{a} have given a complete characterization of the exponential systems which constitute a frame, respectively a Riesz sequence, in $L^2(I)$ (see \cite{cerda} and \cite[Theorem 10]{seip}). Combining Theorem \ref{thm:duality} with these characterizations, one may attempt to determine when $E(\Lambda)$ is just a frame, or just a Riesz sequence, in $L^2(S)$ when $\mes S= |I|$. 

\subsection{}
Several other problems are mentioned in \cite{justmeyer}. In particular, do the results
admit a version for $L^p$ norms, with $p \neq 2$? And what can be said on 
the exponential system  $E(\Lambda)$ when $\Lambda$ is a \emph{non-simple} model set?


\end{document}